\documentclass[12pt]{article}
\usepackage[centertags]{amsmath}
\usepackage{amsfonts}
\usepackage{amssymb}
\usepackage{latexsym}
\usepackage{amsthm}
\usepackage{newlfont}
\usepackage{graphicx}
\usepackage{listings}
\usepackage{booktabs}
\usepackage{abstract}
\date{}

\newlength{\defbaselineskip}
\newcommand{\setlinespacing}[1]%
           {\setlength{\baselineskip}{#1 \defbaselineskip}}

\newcommand{\N}{{\mathbb{N}}}

\newcommand{\actaqed}{\hfill $\actabox$}
{\medskip\noindent \textit{Proof of #1. }}%
{\actaqed \medskip}

\def\C{{\mathcal C}}
\def\cF{{\mathcal F}}

\def \E{\mathcal E}
\def \cE{\mathcal E}
\def \bbE{\mathbb E}

\def \cM{\mathcal M}
\def\R{{\mathbb R}}
\def\Z{\mathbb Z}

\def \T{\mathbb T}

\def\bE{\mathbb E}

\def \<{\langle}
\def\>{\rangle}

\def \La{\Lambda}

\def \e{\varepsilon}
\def \de{\delta}

\def \ff{\varphi}

\def\la{\lambda}

\def\bx{\mathbf x}
\def\by{\mathbf y}
\def\bz{\mathbf z}
\def\bk{\mathbf k}

\def\bw{\mathbf w}

\def\bn{\mathbf n}

\def\bW{\mathbf W}

\def\bE{\mathbf E}
\def\bF{\mathbf F}

\newtheorem{Theorem}{Theorem}[section]
\newtheorem{Lemma}{Lemma}[section]

\newtheorem{Proposition}{Proposition}[section]
\newtheorem{Remark}{Remark}[section]

\newtheorem{Corollary}{Corollary}[section]
\numberwithin{equation}{section}

\newcommand{\be}{\begin{equation}}
\newcommand{\ee}{\end{equation}}

\begin{document}

\title{ Sampling discretization error  for function classes}
\author{  V.N. Temlyakov\thanks{University of South Carolina, Steklov Institute of Mathematics, and Lomonosov Moscow State University.  }}
\maketitle
\begin{abstract}
{The new ingredient of this paper is that we consider infinitely dimensional classes of functions and instead of the relative error setting, which was used in previous papers on norm discretization, we consider the absolute error setting. We demonstrate how known results from two areas of research -- supervised learning theory and numerical integration -- can be used in sampling discretization of the square norm on different function classes.  }
\end{abstract}

\section{Introduction}
This paper is devoted to a study of discretization of the $L_2$ norm of continuous functions.
Recently, in a number of papers (see \cite{VT158}, \cite{VT159}, \cite{DPTT}, \cite{KT168}) a systematic study of the problem of discretization of the $L_q$ norms of elements of finite dimensional subspaces has begun. The first results in this direction were obtained by Marcinkiewicz and 
by Marcinkiewicz-Zygmund (see \cite{Z}) for discretization of the $L_q$ norms of the univariate trigonometric polynomials in 1930s. This is why we call discretization results of this kind the Marcinkiewicz-type theorems. There are different ways to discretize: use coefficients from an expansion with respect to a basis, more generally, use 
linear functionals. We discuss here the way which uses function values at a fixed finite set of points. We call this way of discretization {\it sampling discretization}. 
In the case of finite dimensional subspaces the following problems were discussed in \cite{VT158} and \cite{VT159}. 

{\bf Marcinkiewicz problem.} Let $\Omega$ be a compact subset of $\R^d$ with the probability measure $\mu$. We say that a linear subspace $X_N$ of the $L_q(\Omega):=L_q(\Omega,\mu)$, $1\le q < \infty$, admits the Marcinkiewicz-type discretization theorem with parameters $m$ and $q$ if there exist a set $\{\xi^\nu \in \Omega, \nu=1,\dots,m\}$ and two positive constants $C_j(d,q)$, $j=1,2$, such that for any $f\in X_N$ we have
\be\label{I.1}
C_1(d,q)\|f\|_q^q \le \frac{1}{m} \sum_{\nu=1}^m |f(\xi^\nu)|^q \le C_2(d,q)\|f\|_q^q.
\ee
In the case $q=\infty$ we define $L_\infty$ as the space of continuous on $\Omega$ functions and ask for 
\be\label{I.2}
C_1(d)\|f\|_\infty \le \max_{1\le\nu\le m} |f(\xi^\nu)| \le  \|f\|_\infty.
\ee
We also use a brief way to express the above property: the $\cM(m,q)$ theorem holds for  a subspace $X_N$ or $X_N \in \cM(m,q)$. 

{\bf Marcinkiewicz problem with weights.}  We say that a linear subspace $X_N$ of the $L_q(\Omega)$, $1\le q < \infty$, admits the weighted Marcinkiewicz-type discretization theorem with parameters $m$ and $q$ if there exist a set of knots $\{\xi^\nu \in \Omega\}$, a set of weights $\{\la_\nu\}$, $\nu=1,\dots,m$, and two positive constants $C_j(d,q)$, $j=1,2$, such that for any $f\in X_N$ we have
\be\label{I.5}
C_1(d,q)\|f\|_q^q \le  \sum_{\nu=1}^m \la_\nu |f(\xi^\nu)|^q \le C_2(d,q)\|f\|_q^q.
\ee
Then we also say that the $\cM^w(m,q)$ theorem holds for  a subspace $X_N$ or $X_N \in \cM^w(m,q)$. 
Obviously, $X_N\in \cM(m,q)$ implies that $X_N\in \cM^w(m,q)$. 

{\bf Marcinkiewicz problem with $\e$.} We write $X_N\in \cM(m,q,\e)$ if (\ref{I.1}) holds with $C_1(d,q)=1-\e$ and $C_2(d,q)=1+\e$.  Respectively, 
we write $X_N\in \cM^w(m,q,\e)$ if (\ref{I.5}) holds with $C_1(d,q)=1-\e$ and $C_2(d,q)=1+\e$.

The Marcinkiewicz problem with $\e$ is devoted to looking for a relative error of discretization. 
It is clear that in the setting of a relative error the necessary condition on the number $m$ of sample points is $m\ge N$, where $N$ is the dimension of the subspace. Thus, we cannot work in the relative error setting for an infinitely dimensional class of functions. The new ingredient of this paper is that we consider infinitely dimensional classes of functions and instead of the relative error setting we consider the absolute error setting. We formulate it explicitly. 

{\bf Sampling discretization with absolute error.} Let $W\subset L_q(\Omega,\mu)$, $1\le q<\infty$, be a class of continuous on $\Omega$ functions. We are interested in estimating 
the following optimal errors of discretization of the $L_q$ norm of functions from $W$
$$
er_m(W,L_q):= \inf_{\xi^1,\dots,\xi^m} \sup_{f\in W} \left|\|f\|_q^q - \frac{1}{m}\sum_{j=1}^m |f(\xi^j)|^q\right|,
$$
$$
er_m^o(W,L_q):= \inf_{\xi^1,\dots,\xi^m;\la_1,\dots,\la_m} \sup_{f\in W} \left|\|f\|_q^q - \sum_{j=1}^m \la_j |f(\xi^j)|^q\right|.
$$

It will be convenient for us to use the following notation. For  given sets $\xi:=\{\xi^j\}_{j=1}^m$ of sampling points and $\La:=\{\la_j\}_{j=1}^m$ of weights we write
$$
er(f,\xi,L_q):=  \left|\|f\|_q^q - \frac{1}{m}\sum_{j=1}^m |f(\xi^j)|^q\right|,
$$
$$
er(f,\xi,\La,L_q):=\left|\|f\|_q^q - \sum_{j=1}^m \la_j |f(\xi^j)|^q\right|.
$$
In this paper we only discuss in detail the case $q=2$. For this reason, in case $q=2$ we drop 
$L_q$ from notation: $er(f,\xi):=er(f,\xi,L_2)$, $er(f,\xi,\La):=er(f,\xi,\La,L_2)$. 

In this paper we demonstrate how known results from two areas of research -- supervised learning theory and numerical integration -- can be used in sampling discretization in $L_2$. 
We now formulate two typical results obtained in the paper. In Section \ref{pa} (see Theorem \ref{BT5}) we obtain the following result (see Section \ref{pa} for the definition of entropy numbers).

\begin{Theorem}\label{IT1} Assume that a class of real functions $W$ is such that for all $f\in W$ we have $\|f\|_\infty \le M$ with some constant $M$. Also assume that the entropy numbers of $W$ in the uniform norm $L_\infty$ satisfy the condition
$$
  \e_n(W) \le Cn^{-r},\qquad r\in (0,1/2).
$$
Then
$$
er_m(W):= er_m(W,L_2)  \le Km^{-r}.
$$
\end{Theorem}

Theorem \ref{IT1} is a rather general theorem, which connects the behavior of absolute errors of discretization with the rate of decay of the entropy numbers. This theorem is derived in Section \ref{pa} from known results in supervised learning theory. It is well understood in learning theory (see, for example, \cite{Tbook}, Ch.4) that the entropy numbers of the class 
of priors (regression functions) is the right characteristic in studying the regression problem. 
We impose a restriction $r<1/2$ in Theorem \ref{IT1} because the probabilistic technique from the supervised learning theory has a natural limitation to $r\le 1/2$. It would be interesting to understand if Theorem \ref{IT1} holds for $r\ge 1/2$. Also, it would be interesting to obtain an analog of Theorem \ref{IT1} for discretization in $L_q$, $1\le q<\infty$, norm.

For classes of smooth functions we obtained in Section \ref{sc} error bounds, which do not have a restriction on smoothness $r$. We proved there (see Theorems \ref{AT1}, \ref{CT2}, and inequality (\ref{C6})) the following bounds for the class $\bW^r_2$ of functions on $d$ variables with bounded in $L_2$ mixed derivative (see Section \ref{sc} for a rigorous definition of this class). 

\begin{Theorem}\label{IT2} Let $r>1/2$ and $\mu$ be the Lebesgue measure on $[0,2\pi]^d$. Then
$$
er_m^o(\bW^r_2,L_2) \asymp m^{-r}(\log m)^{(d-1)/2}.
$$
\end{Theorem}
The proof of upper bound in Theorem \ref{IT2} is given in Section \ref{sc}. It  uses deep results from numerical integration of functions from $\bW^r_2$. The lower bound in Theorem \ref{IT2} is proved in Section \ref{lb}.
 It is well known (see \cite{DPTT}) how numerical integration can be used in a problem of exact discretization of $L_q$ norm of elements of finite dimensional subspaces 
in case $q$ is an even integer.  We present some results on sampling discretization of the $L_q$ norm, $q$ is an even integer, with absolute error at the end of the paper. 

As we have mentioned above results from the supervised learning theory are used in the proof of Theorem \ref{IT1}. For the reader's convenience we present a brief introduction to the supervised learning theory and formulate results which we use.

\section{Probabilistic approach}
\label{pa}

\subsection{Some classical results}\label{scr}
We begin with the well known Monte Carlo method. For the readers convenience we present here the classical argument on the error bound for the Monte Carlo method. Let $\Omega$ be a bounded domain of $\R^d$. 
Consider a real function $f\in L_2(\Omega):=L_2(\Omega,\mu)$ with respect to a probability measure $\mu$. Define $\Omega^m:= \Omega\times\cdots\times\Omega$ and $\mu^m:=\mu \times\cdots\times \mu$. For $\bx^j\in \Omega$ denote $\bz:= (\bx^1,\dots,\bx^m)\in \Omega^m$ and for $g\in L_1(\Omega^m, \mu^m)$
$$
\bbE(g):= \int_{\Omega^m}g(\bz)d\mu^m.
$$
Then we have for $f\in L_2(\Omega,\mu)$
$$
\bbE\left(\left(\int_\Omega f d\mu - \frac{1}{m}\sum_{j=1}^m f(\bx^j)\right)^2\right) 
$$
$$
= \bbE\left(\left(\int_\Omega f d\mu\right)^2 - \frac{2}{m}\left(\int_\Omega f d\mu\right)\left(\sum_{j=1}^m f(\bx^j)\right) +\frac{1}{m^2}\sum_{i,j=1}^m f(\bx^i) f(\bx^j)\right)
$$
$$
= \left(\int_\Omega f d\mu\right)^2 -2\left(\int_\Omega f d\mu\right)^2 + \frac{m(m-1)}{m^2}\left(\int_\Omega f d\mu\right)^2 + \frac{1}{m}\int_\Omega f^2d\mu
$$
$$
=\frac{1}{m} \left(\int_\Omega f^2d\mu -\left(\int_\Omega f d\mu\right)^2\right) \le \frac{1}{m} \int_\Omega f^2d\mu = \|f\|_2^2/m.
$$

In particular, the above argument, which uses expectation $\bbE(\cdot)$, guarantees existence of a cubature formula $(\xi,\La)$, $\xi:=\{\xi^j\}_{j=1}^m$, $\La:=\{\la_j\}_{j=1}^m$, $\Lambda_m(\cdot,\xi):= \sum_{j=1}^m \la_j f(\xi^j)$ such that
\be\label{B1}
\left|\int_\Omega fd\mu -\Lambda_m(f,\xi)\right| \le m^{-1/2}\|f\|_2.
\ee
The use of expectation does not provide a good bound on probability to guarantee a tight error bound alike (\ref{B1}). The concentration measure inequalities, which we formulate momentarily, provide a  very good bound on probability under some extra assumptions on $f$. Under condition $\|f\|_\infty \le M$ the Hoeffding's inequality (see, for instance, \cite{Tbook}, p.197) gives
\be\label{B2}
\mu^m\left\{\bz: \left|\int_\Omega fd\mu - \frac{1}{m}\sum_{j=1}^m f(\bx^j)\right|\ge \eta\right\}\le 2\exp\left(-\frac{m\eta^2}{8M^2}\right).
\ee
The Bernstein's inequality (see, for instance, \cite{Tbook}, p.198) gives the following bound 
under conditions $\|f\|_\infty\le M_\infty$ and $\|f\|_2 \le M_2$
\be\label{B3}
\mu^m\left\{\bz: \left|\int_\Omega fd\mu - \frac{1}{m}\sum_{j=1}^m f(\bx^j)\right|\ge \eta\right\}\le 2\exp\left(-\frac{m\eta^2}{2(M_2^2+2M_\infty\eta/3)}\right).
\ee

The above inequalities (\ref{B2}) and (\ref{B3}) can be used directly for proving existence of good cubature formulas for function classes containing finite number of elements. Denote $|W|$ cardinality of a set $W$. Assume that for all $f\in W$ we have $\|f\|_\infty \le M$. Then, the Hoeffding's inequality (\ref{B2}) gives
\be\label{B4}
\mu^m\left\{\bz: \sup_{f\in W}\left|\int_\Omega fd\mu - \frac{1}{m}\sum_{j=1}^m f(\bx^j)\right|\le \eta\right\}\ge 1- 2|W|\exp\left(-\frac{m\eta^2}{8M^2}\right).
\ee
Thus, in case the right hand side of (\ref{B4}) is positive, inequality (\ref{B4}) guarantees existence of a good cubature formula for the whole class $W$.  

\subsection{Some results from supervised learning theory}\label{lt}

In our further discussion we are interested in discretization of the $L_2$ norm of real functions from a given function class $W$. It turns out that 
this problem is closely related to some problems from supervised learning theory. We give a brief introduction to these problems. This is a vast area of research with a wide range of different settings. In this subsection we only discuss a development of a setting from \cite{CS} (see \cite{Tbook}, Ch.4, for detailed discussion). 

Let $X\subset\R^d$, $Y\subset \R$ be Borel sets, $\rho$ be a Borel probability measure on $Z:=X\times Y$.  For $f:X\to Y$ define {\it the error} 
$$
\cE(f)  :=\int_Z(f(\bx)-y)^2 d\rho.
$$
 
Consider   $\rho_X$ - the marginal probability  measure on $X$ (for $S\subset X$, $\rho_X (S) = \rho(S\times Y)  $). Define
$$
f_\rho(\bx) := \bbE(y|\bx)
$$
to be a conditional expectation of $y$.
The function $f_\rho$ is known in statistics as the {\it regression function} of $\rho$. It is clear that if $f_\rho\in L_2(\rho_X)$ then it minimizes the error $\cE(f)$ over all $f\in L_2(\rho_X)$: $\cE(f_\rho)\le \cE(f)$, $f\in L_2(\rho_X)$. Thus, in the sense of error $\cE(\cdot)$ the regression function $f_\rho$ is the best to describe the relation between inputs $\bx\in X$ and outputs $y\in Y$. The goal is to find an estimator $f_\bz$, on the base of given data $\bz:=((\bx^1,y_1),\dots,(\bx^m,y_m))$ that approximates $f_\rho$  well with high probability. We assume that $(\bx^i,y_i)$, $i=1,\dots,m$ are independent and distributed according to $\rho$.   We measure the error between $f_\bz$ and $f_\rho$ in the $L_2(\rho_X)$ norm. 

We note that a standard setting in the distribution-free theory of regression (see \cite{GKKW}) involves the expectation $\bbE(\|f_\rho- f_\bz\|_{L_2(\rho_X)}^2)$ as a measure of quality of an estimator. An important new feature of  the setting in learning theory formulated in \cite{CS} (see \cite{Tbook} for detailed discussion) is the following.  They propose to 
study systematically the probability distribution function
$$
\rho^m\{\bz:\|f_\rho-f_\bz\|_{L_2(\rho_X)}\ge\eta\}
$$
instead of the expectation. 

There are several important ingredients in mathematical formulation of the learning problem. In our formulation we follow the way that has become standard in approximation theory and based on the concept of {\it optimal method}. 

We begin with a class ${\mathcal M}$ of admissible measures $\rho$. Usually, we impose restrictions on $\rho$ in the form of restrictions on the regression function $f_\rho$: $f_\rho\in \Theta$. Then the first step is to find an optimal estimator for a given class $\Theta$ of priors (we assume $f_\rho \in \Theta$).  In regression theory a usual way to evaluate performance of an estimator $f_\bz$ is by studying its convergence in expectation, i.e. the rate of decay of the quantity $\bbE(\|f_\rho-f_\bz\|^2_{L_2(\rho_X)})$ as the sample size $m$ increases. Here the expectation is taken with respect to the product measure $\rho^m$ defined on $Z^m$. We note that ${\mathcal E}(f_\bz)-{\mathcal E}(f_\rho) = \|f_\bz-f_\rho\|_{L_2(\rho_X)}^2$. As we already mentioned above
a more accurate and more delicate way of evaluating performance of $f_\bz$ has been pushed forward in \cite{CS}.  We concentrate on a discussion of results on  the probability distribution function. 
 
An important question in finding an optimal $f_\bz$ is the following. How to describe the class $\Theta$ of priors? In other words, what characteristics of $\Theta$ govern, say, the optimal rate of decay of $\bbE(\|f_\rho-f_\bz\|^2_{L_2(\rho_X)})$ for $f_\rho\in \Theta$?
Previous and recent works in statistics and learning theory (see, for instance, \cite{DKPT}, \cite{VT110}, and \cite{Tbook}, Ch.4) indicate that the compactness characteristics of $\Theta$ play a fundamental role in the above problem. It is convenient for us to express compactness of $\Theta$ in terms of the entropy numbers. We discuss the classical concept of entropy. We note that some other concepts of entropy, for instance, entropy with bracketing, proved to be useful in the theory of empirical processes and nonparametric statistics (see \cite{VG},   \cite{Va}). There is a concept of $VC$ dimension that plays a fundamental role in the problem of pattern recognition and classification \cite{Va}. This concept is also useful in describing compactness characteristics of sets.  

For a compact subset $\Theta$ of a Banach space $B$ we define the entropy numbers as follows
$$
\e_n(\Theta,B) := \inf\{\e: \exists f_1,\dots,f_{2^n}\in \Theta: \Theta\subset \cup_{j=1}^{2^n} (f_j+\e U(B))\}
$$
where $U(B)$ is the unit ball of a Banach space $B$. We denote $N(\Theta,\e,B)$ the covering number that is the minimal number of balls of radius $\e$ with centers in $\Theta$ needed for covering $\Theta$. The corresponding $\e$-net is denoted by ${\mathcal N}_\e(\Theta,B)$.   
  
In   this subsection we always assume that the measure $\rho$ satisfies the condition $|y|\le M$ (or a little weaker $|y|\le M$ a.e. with respect to $\rho_X$) with some fixed $M$. Then it is clear that for $f_\rho$ we have $|f_\rho(\bx)|\le M$ for all $\bx$ (for almost all $\bx$). Therefore, it is natural to assume that a class $\Theta$ of priors where $f_\rho$ belongs is embedded into the $\C(X)$-ball ($L_\infty$-ball) of radius $M$.  

We define the {\it empirical error} of $f$ as
$$
\E_\bz(f):= \frac{1}{m}\sum_{i=1}^m(f(\bx^i)-y_i)^2.
$$
Let $f\in L_2(\rho_X)$. The {\it defect function} of $f$ is
$$
L_\bz(f) := L_{\bz,\rho}(f) := \E(f)-\E_\bz(f);\quad \bz=(z_1,\dots,z_m),\quad z_i=(\bx^i,y_i).
$$
We are interested in estimating $L_\bz(f)$ for functions $f$ coming from a given class $W$. 
We begin with the case $B$ being 
 $\C(X)$ the space of functions continuous on a compact subset $X$ of $\R^d$ with the norm
$$
\|f\|_\infty:= \sup_{\bx\in X}|f(\bx)|.
$$
  We use the abbreviated notations
$$
N(W,\e):= N(W,\e,\C);\quad \e_n(W):= \e_n(W,\C).
$$
The following well known theorem (see, for instance, \cite{Tbook}, p.211) shows how compactness characteristics of $W$ can be used in estimating the defect function.
 \begin{Theorem}\label{BT1} Let $W$ be a compact subset of $\C(X)$. Assume that  $\rho$ and $W$ satisfy the following condition. Let $M>0$ and for all $f\in W$ we have $|f(\bx)-y| \le M$ a.e. Then, for all $\e>0$
\begin{equation}\label{B5}
\rho^m\{\bz:\sup_{f\in W}|L_\bz(f)|\le\e\} \ge1- N(W,\e/(8M))2\exp(-\frac{m\e^2}{8(\sigma^2+M^2\e/6)}).  
\end{equation}
Here $\sigma^2:=\sigma^2(W) := \sup_{f\in W}\sigma^2((f(\bx)-y)^2)$ and $\sigma^2(g)$ is the variance of a random variable $g$.
\end{Theorem}
\begin{Remark}\label{BR1} In general we cannot guarantee that the set \newline $\{\bz:\sup_{f\in W}|L_\bz(f)|\ge\eta\}$ is $\rho^m$-measurable. In such a case the relation (\ref{B5}) and further relations of this type are understood in the sense of outer measure associated with the $\rho^m$. For instance, for (\ref{B5}) this means that there exists $\rho^m$-measurable set $G$ such that $\{\bz:\sup_{f\in W}|L_\bz(f)|\ge\eta\}\subset G$ and (\ref{B5}) holds for $G$.
\end{Remark}

We note that the above theorem is related to the concept of
 the Glivenko-Cantelli sample complexity of a class $\Phi$ with accuracy $\e$ and confidence $\de$:
$$
S_\Phi(\e,\de):= \min\{n:\quad\text{ for \, all}\quad m\ge n, \quad \text{for \, all}\quad \rho 
$$
$$
 \rho^m\{\bz=(z_1,\dots,z_m):\sup_{\phi\in \Phi}|\int_Z \phi d\rho-\frac{1}{m}\sum_{i=1}^m\phi(z_i)|\ge\e\} \le \de\}.
$$
In order to see that we define $z_i:=(\bx^i,y_i)$, $i=1,\dots,m$; $\phi(\bx,y):=(f(\bx)-y)^2$; $\Phi:=\{(f(\bx)-y)^2, f\in W\}$. One can find a survey of results on the Glivenko-Cantelli sample complexity in \cite{Me} and find results and the corresponding historical remarks related to Theorem \ref{BT1}  in \cite{GKKW}.

We now formulate two theorems, which provide somewhat more delicate estimates for the defect function (see \cite{Tbook}, pp. 213--217). We assume that $\rho$ and $W$ satisfy the following condition. 
\begin{equation}\label{B6}
\text{For\, all}\quad f\in W,\quad f:X\to Y\quad \text{ is\, such\, that} \quad |f(\bx)-y| \le M \quad \text{ a.e.}  
\end{equation}
The following Theorems \ref{BT2}, \ref{BT3} and Corollaries \ref{BC1}, \ref{BC2} are from \cite{VT98} (see also \cite{Tbook}, section 4.3.3, p.213).

\begin{Theorem}\label{BT2} Assume that $\rho$, $W$ satisfy (\ref{B6}) and $W$ is such that 
\begin{equation}\label{B7}
\sum_{n=1}^\infty n^{-1/2}\e_n(W) <\infty.  
\end{equation}
Then for $m\eta^2\ge 1$ we have
$$
\rho^m\{\bz:\sup_{f\in W}|L_\bz(f)|\ge \eta\} \le C(M,\e(W))\exp(-c(M)m\eta^2).
$$
with $C(M,\e(W))$ that may depend on $M$ and $\e(W):=\{\e_n(W,\C)\}$; $c(M)$ may depend only on $M$.
\end{Theorem}

\begin{Theorem}\label{BT3} Assume that $\rho$, $W$ satisfy (\ref{B6}) and $W$ is such that
$$
\sum_{n=1}^\infty n^{-1/2}\e_n(W) =\infty.
$$
For $\eta>0$ define $J:=J(\eta/M)$ as the minimal $j$ satisfying $\e_{2^j}\le \eta/(8M)$ and
$$
S_J:= \sum_{j=1}^J2^{(j+1)/2}\e_{2^{j-1}}.
$$
Then for $m$, $\eta$ satisfying $m(\eta/S_J)^2 \ge 480M^2$ we have
$$
\rho^m\{\bz:\sup_{f\in W} |L_\bz(f)|\ge \eta\} \le C(M,\e(W))\exp(-c(M)m(\eta/S_J)^2).
$$
\end{Theorem}

\begin{Corollary}\label{BC1} Assume $\rho$, $W$ satisfy (\ref{B6}) and $\e_n(W)\le Dn^{-1/2}$.
Then for $m$, $\eta$ satisfying $m(\eta/(1+\log(M/\eta)))^2 \ge C_1(M,D)$ we have
$$
\rho^m\{\bz:\sup_{f\in W} |L_\bz(f)|\ge \eta\} \le C(M,D)\exp(-c(M,D)m(\eta/(1+\log (M/\eta)))^2).
$$
\end{Corollary}
\begin{Corollary}\label{BC2} Assume $\rho$, $W$ satisfy (\ref{B6}) and 
$\e_n(W)\le Dn^{-r}$, $r\in (0,1/2)$.
Then for $m$, $\eta$ satisfying $m \eta^{1/r} \ge C_1(M,D,r)$ we have
$$
\rho^m\{\bz:\sup_{f\in W} |L_\bz(f)|\ge \eta\} \le C(M,D,r)\exp(-c(M,D,r)m\eta^{1/r} ).
$$
\end{Corollary}

\subsection{Application of supervised learning theory for discretization}\label{alt}

Settings for the supervised learning problem and the discretization problem are different. 
In the supervised learning problem we are given a sample $\bz$ and we want to approximately recover the regression function $f_\rho$. It is important that we do not know 
$\rho$. We only assume that we know that $f_\rho\in \Theta$. In the discretization of the $L_q$, $1\le q<\infty$, norm we assume that $f\in W$ and the probability measure $\mu$ is known. We want to find a discretization set $\xi=\{\bx^j\}_{j=1}^m$, which is good for the whole class $W$. However, the technique, based on the defect function, for solving the supervised learning problem can be used for solving the discretization problem. We now explain this in detail. Let us consider a given function class $W$ of real functions, defined on $X=\Omega$.
Suppose that the probability measure $\rho$ is such that $\rho_X=\mu$ and for all $\bx \in X$ 
we have $y=0$. Then for the defect function we have
$$
L_\bz(f) =\int_X f^2d\mu -\frac{1}{m}\sum_{j=1}^m f(\bx^j)^2 =: L^2_{(\bx^1,\dots,\bx^m)}(f)
$$
and 
$$
\rho^m\{\bz:\sup_{f\in W} |L_\bz(f)|\ge \eta\} = \mu^m\{\bw:\sup_{f\in W} |L_\bw^2(f)|\ge \eta\}.
$$
Moreover, condition (\ref{B6}) is satisfied with $M$ such that for all $f\in W$ we have $\|f\|_\infty \le M$. The above argument shows that we can derive results on discretization of the $L_2$ norm directly from the corresponding results from learning theory. We assume that $W$ satisfies the following condition:
\be\label{B8}
f\in W\quad \Rightarrow \quad \|f\|_\infty \le M.
\ee

Theorem \ref{BT2} implies the following result. 

\begin{Theorem}\label{BT4} Assume that $W$ satisfies (\ref{B8}) and the condition
$$
\sum_{n=1}^\infty n^{-1/2} \e_n(W) < \infty.
$$
Then there exists a constant $K$ such that for any $m$ there is a set of points $\xi=\{\xi^1,\dots,\xi^m\}$ such that for all $f\in W$ 
\be\label{B9}
er(f,\xi) = \left|\|f\|_2^2 - \frac{1}{m}\sum_{j=1}^m f(\xi^j)^2 \right| \le Km^{-1/2}.
\ee
In particular, if $\e_n(W) \le C_1 n^{-r}$, $r>1/2$, then $er_m(W,L_2) \le C(r,C_1) m^{-1/2}$.
\end{Theorem}
 
Corollary \ref{BC2} implies the following result.

\begin{Theorem}\label{BT5} Assume that $W$ satisfies (\ref{B8}) and the condition
$$
  \e_n(W) \le C_1 n^{-r},\qquad r\in (0,1/2).
$$
Then there exists a constant $K$ such that for any $m$ there is a set of points $\xi=\{\xi^1,\dots,\xi^m\}$ such that for all $f\in W$ 
\be\label{B10}
er(f,\xi) = \left|\|f\|_2^2 - \frac{1}{m}\sum_{j=1}^m f(\xi^j)^2 \right| \le Km^{-r}.
\ee
\end{Theorem}

\section{Smoothness classes}
\label{sc}

We begin with a very simple general observation on a connection between norm discretization and numerical integration. 

{\bf Quasi-algebra property.} We say that a function class $W$ has the quasi-algebra property if  there exists a constant $a$ such that for any $f,g\in W$ we have $fg/a \in W$.

We now formulate a simple statement, which gives a connection between numerical integration and discretization of the $L_2$ norm. 

\begin{Proposition}\label{AP2} Suppose that a function class $W$ has the quasi-algebra property and for any $f\in W$ we have for the complex conjugate function $\bar f\in W$. 
  Then for a cubature formula $\Lambda_m(\cdot,\xi)$ we have: for any $f\in W$
$$
|\|f\|_2^2 -\Lambda_m(|f|^2,\xi)| \le a\sup_{g\in W} \left|\int_\Omega gd\mu - \Lambda_m(g,\xi)\right|.
$$
\end{Proposition}
Obviously, an analog of Proposition \ref{AP2} holds for the $L_q$ norm in case $q$ is an even natural number. We formulate it as a remark.
\begin{Remark}\label{AR1} Suppose that a function class $W$ has the quasi-algebra property and for any $f\in W$ we have for the complex conjugate function $\bar f\in W$. Let $q$ be an even number.
  Then for a cubature formula $\Lambda_m(\cdot,\xi)$ we have: for any $f\in W$
$$
|\|f\|_q^q -\Lambda_m(|f|^q,\xi)| \le C(a,q)\sup_{g\in W} \left|\int_\Omega gd\mu - \Lambda_m(g,\xi)\right|.
$$
\end{Remark}

In this section we discuss some classical classes of smooth periodic functions. We begin with a general scheme and then give two concrete examples. Let $F\in L_1(\T^d)$ be such that 
$\hat F(\bk) \neq 0$ for all $\bk \in \Z^d$, where
$$
\hat F(\bk) := \cF(F,\bk):=(2\pi)^{-d}\int_{\T^d} F(\bx)e^{-i(\bk,\bx)}d\bx.
$$
Consider the space
$$
W^F_2 := \left\{f: f(\bx) = J_F(\ff)(\bx):= (2\pi)^{-d}\int_{\T^d} F(\bx-\by)\ff(\by)d\by,\quad \|\ff\|_2<\infty\right\}.
$$
For $f\in W^F_2$ we have $\hat f(\bk)=\hat F(\bk)\hat \ff(\bk)$ and, therefore, our assumption 
$\hat F(\bk) \neq 0$ for all $\bk \in \Z^d$ implies that function $\ff$ is uniquely defined by $f$.
Introduce a norm on $W^F_2$ by 
$$
\|f\|_{W^F_2} := \|\ff\|_2, \qquad f=J_F(\ff). 
$$
For convenience, with a little abuse of notation we will use notation $W^F_2$ for the unit ball of the space $W^F_2$. We are interested in the following question. Under what conditions on $F$ the fact that $f,g \in W^F_2$ implies that $fg\in W^F_2$ and 
$$
\|fg\|_{W^F_2}\le C_0\|f\|_{W^F_2}\|g\|_{W^F_2} ?
$$
In other words: Which properties of $F$ guarantee that the class $W^F_2$ has the quasi-algebra property?
We give a simple sufficient condition. 

\begin{Proposition}\label{AP1} Suppose that for each $\bn\in\Z^d$ we have
\be\label{A1}
\sum_{\bk\in Z^d} |\hat F(\bk)\hat F(\bn-\bk)|^2 \le C_0^2|\hat F(\bn)|^2.
\ee
Then, for any $f,g\in W^F_2$ we have $fg\in W^F_2$ and
$$
\|fg\|_{W^F_2}\le C_0\|f\|_{W^F_2}\|g\|_{W^F_2}.
$$
\end{Proposition}
\begin{proof} Let $f=J_F(\ff)$ and $g=J_F(\psi)$. Then
$$
\cF(fg,\bn) = \sum_{\bk\in\Z^d} \hat f(\bk)\hat g(\bn-\bk) =  \sum_{\bk\in\Z^d} \hat F(\bk) \hat \ff(\bk)\hat F(\bn-\bk) \hat \psi(\bn-\bk).
$$
By Cauchy inequality
$$
\|fg\|^2_{W^F_2} = \sum_{\bn\in \Z^d} |\cF(fg,\bn)|^2 |\hat F(\bn)|^{-2} 
$$
$$
\le 
\sum_{\bn\in \Z^d} |\hat F(\bn)|^{-2} \left(\sum_{\bk\in\Z^d} |\hat \ff(\bk)\hat \psi(\bn-\bk)|^2\right)\left(\sum_{\bk\in\Z^d} |\hat F(\bk)\hat F(\bn-\bk)|^2\right)
$$
$$
\le C_0^2 \sum_{\bn\in \Z^d} \sum_{\bk\in\Z^d} |\hat \ff(\bk)|^2|\hat \psi(\bn-\bk)|^2 \le
C_0^2 \|f\|^2_{W^F_2} \|g\|^2_{W^F_2}.
$$
This proves Proposition \ref{AP1}.

\end{proof}

As an example consider the class $\bW^r_2$ of functions with bounded mixed derivative. 
By the definition $\bW^r_2 := W^{F_r}_2$ with function $F_r(\bx)$ defined as follows.
For a number $k\in \Z$ denote $k^* := \max(|k|,1)$. Then for $r>0$ we define $F_r$ by its Fourier coefficients
\be\label{A2}
\hat F_r(\bk) = \prod_{j=1}^d (k_j^*)^{-r}.
\ee

\begin{Lemma}\label{AL1} Function $F=F_r$ with $r>1/2$ satisfies condition (\ref{A1}) with 
$C_0 = C(r,d)$. 
\end{Lemma}
\begin{proof} Relation (\ref{A2}) implies that it is sufficient to prove Lemma \ref{AL1} in case $d=1$. For $n\in\Z$ we have
$$
\sum_{k\in\Z} (k^*)^{-2r}((n-k)^*)^{-2r} \le \sum_{k:|n-k|\ge |n|/2} (k^*)^{-2r}((|n|/2)^*)^{-2r}
$$
$$
+ \sum_{k:|n-k|< |n|/2} ((|n|/2)^*)^{-2r} ((n-k)^*)^{-2r} \le C(r) (n^*)^{-2r}.
$$

\end{proof}

Lemma \ref{AL1} and Proposition \ref{AP1} imply that the class $\bW^r_2$ has the quasi-algebra property. 
 We now illustrate how a combination of Proposition \ref{AP2} and known results on numerical integration gives results on discretization. We discuss classes of periodic functions. We begin with the case of functions of two variables. Let $\{b_n\}_{n=0}^{\infty}$, $b_0=b_1 =1$, $b_n = b_{n-1}+b_{n-2}$,
$n\ge 2$, --
be the Fibonacci numbers. For the continuous functions of two
variables which are $2\pi$-periodic in each variable we define
cubature formulas
$$
\Phi_n(f) :=b_n^{-1}\sum_{\mu=1}^{b_n}f\bigl(2\pi\mu/b_n,
2\pi\{\mu b_{n-1} /b_n \}\bigr),
$$
which will be called the {\it Fibonacci cubature formulas}. In this
definition $\{a\}$ is the fractional part of the number $a$. For a function class $\bF$
denote
$$
\Phi_n(\bF) := \sup_{f\in \bF} |\Phi_n(f) - \hat f(\mathbf 0)|.
$$
The following result is known (see \cite{VTbookMA}, p.275)
\be\label{A3}
\Phi_n (\bW_{2}^r )\asymp b_n^{-r} (\log b_n)^{1/2},\qquad r>1/2.
\ee

Combining (\ref{A3}) with Proposition \ref{AP2} we obtain the following discretization result.

\begin{Theorem}\label{ATFib} Let $d=2$, $r>1/2$ and $\mu$ be the Lebesgue measure on $[0,2\pi]^2$. Then
$$
er_m(\bW^r_2,L_2) \le C(r) m^{-r}(\log m)^{1/2}.
$$
\end{Theorem}

Using the Korobov cubature formulas (see \cite{VTbookMA}, section 6.6, p.284) instead of the Fibonacci cubature formulas one obtains the following discretization result (see \cite{VTbookMA}, p.287).

\begin{Theorem}\label{ATKor} Let $r>1$ and $\mu$ be the Lebesgue measure on $[0,2\pi]^d$. Then
$$
er_m(\bW^r_2,L_2) \le C(r,d) m^{-r}(\log m)^{r(d-1)}.
$$
\end{Theorem}

As a direct corollary of Proposition \ref{AP2}, Lemma \ref{AL1} and known results on optimal error bounds for numerical integration for classes $\bW^r_2$ (see, for instance,  \cite{VTbookMA}, section 6.7, p.289 and \cite{DTU}, Ch.8)
we obtain the following theorem.

\begin{Theorem}\label{AT1} Let $r>1/2$ and $\mu$ be the Lebesgue measure on $[0,2\pi]^d$. 
Then 
$$
er_m^o(\bW^r_2,L_2) \le C(r,d)m^{-r}(\log m)^{(d-1)/2}.
$$
\end{Theorem}

Consider the Korobov class $\bE^r$. For $r>1$ define a class of continuous periodic functions 
$$
\bE^r := \{ f: |\hat f(\bk)| \le \prod_{j=1}^d (k_j^*)^{-r}\}.
$$
Lemma \ref{AL1} implies that there exists $C_0=C(r,d)$ such that for any $f,g\in\bE^r$ we have $fg/C_0 \in \bE^r$. Thus, class $\bE^r$ has the quasi-algebra property. Using the Korobov cubature formulas (see \cite{VTbookMA}, section 6.6, p.284) we obtain the following discretization result (see \cite{VTbookMA}, p.286).

\begin{Theorem}\label{ATE} Let $r>1$ and $\mu$ be the Lebesgue measure on $[0,2\pi]^d$. Then
$$
er_m(\bE^r,L_2) \le C(r,d) m^{-r}(\log m)^{r(d-1)}.
$$
\end{Theorem}

We introduce some notation, which we use here and in Section \ref{lb}.
Let $\Omega$ be a compact subset of $\R^d$ and $\mu$ be a probability measure on $\Omega$. Denote
$$
I_\mu(f) := \int_\Omega fd\mu
$$
and consider a cubature formula $(\xi,\La)$
$$
\La_m(f,\xi) = \sum_{j=1}^m \la_jf(\xi^j).
$$
For a function class $W\subset \C(\Omega)$ consider the best error of numerical integration by cubature formulas with $m$ knots:
$$
\kappa_m(W) := \inf_{(\xi,\La)} \sup_{f\in W} \left|I_\mu(f) - \La_m(f,\xi)\right|.
$$

The following result is known (see \cite{Fro2} and \cite{Byk})
\be\label{A4}
\kappa_m(\bE^r) \le C(r,d)m^{-r}(\log m)^{d-1}.
\ee
Therefore, Proposition \ref{AP2} and inequality (\ref{A4}) imply the following theorem.

\begin{Theorem}\label{AT2}  Let $r>1$ and $\mu$ be the Lebesgue measure on $[0,2\pi]^d$. Then
$$
er_m^o(\bE^r,L_2) \le C(r,d) m^{-r}(\log m)^{d-1}.
$$
\end{Theorem}

We discussed above the case of $L_2$ norm. Clearly Remark \ref{AR1} allows us to obtain 
versions of the above theorems for $L_q$ with even $q$. 

\begin{Remark} The above Theorems \ref{ATFib} -- \ref{AT2} hold for the $L_q$ norm with $q$ even instead of $L_2$ norm with constants allowed to depend on $q$. 
\end{Remark}

\section{Some lower bounds for the norm discretization}
\label{lb}

In this section we show on the example of discretization of the $L_1$ and $L_2$ norms that 
the problem of discretization is a more difficult problem than the problem of numerical integration.  First, we discuss the lower bounds in discretization of the $L_1$ norm. 

\begin{Theorem}\label{CT1} Let $V\subset \C(\Omega)$ be a Banach space and 
$W:=\{f:\|f\|_V\le 1\}$ be its unit ball. Then for any $m\in \N$ we have
$$
er_m^o(W,L_1) \ge \kappa_m(W).
$$
\end{Theorem}
\begin{proof} We have for a given cubature formula $(\xi,\La)$
\be\label{C1}
\sup_{f\in W} \left|I_\mu(f) - \La_m(f,\xi)\right| = \left\|I_\mu - \sum_{j=1}^m \la_j \delta_{\xi^j}\right\|_{V'},
\ee
where $V'$ is a dual (conjugate) to $V$ Banach space and $\delta_{\xi^j}$ are the Dirac delta functions. 
By the Nikol'skii duality theorem (see, for instance, \cite{VTbookMA}, p.509) we find
\be\label{C2}
\inf_\La   \left\|I_\mu - \sum_{j=1}^m \la_j \delta_{\xi^j}\right\|_{V'} = \sup_{f\in W: f(\xi^j)=0,j=1,\dots,m} |I_\mu(f)|.
\ee
It follows from the definition of $\kappa_m(W)$ and from relations (\ref{C1}) and (\ref{C2}) that
for any $\xi$ 
\be\label{C3}
\sup_{f\in W: f(\xi^j)=0,j=1,\dots,m} |I_\mu(f)| \ge \kappa_m(W).
\ee
Next, for $f\in W$ such that $f(\xi^j)=0$, $j=1,\dots,m$, we get for any $\La$
\be\label{C4}
er(f,\xi,\La,L_1) = \|f\|_1 = \int_{\Omega} |f|d\mu \ge |I_\mu(f)|.
\ee
Obviously, (\ref{C4}) and (\ref{C3}) imply the conclusion of Theorem \ref{CT1}.

\end{proof}

We now proceed to the case of $L_2$ norm. In this case it is convenient for us to consider real functions. Assume that a class of real functions $W\subset \C(\Omega)$ has the following extra property. 

{\bf Property A.} For any $f\in W$ we have $f^+ := (f+1)/2 \in W$ and $f^- := (f-1)/2 \in W$.

In particular, this property is satisfied if $W$ is a convex set containing function $1$. 

\begin{Theorem}\label{CT2} Suppose $W\subset \C(\Omega)$ has Property A. Then for any $m\in \N$ we have
$$
er_m^o(W,L_2) \ge \frac{1}{2}\kappa_m(W).
$$
\end{Theorem}
\begin{proof} For any cubature formula $(\xi,\La)$ we have
\be\label{C5}
er(f^+,\xi,\La) - er(f^-,\xi,\La) = I_\mu(f) - \La_m(f,\xi).
\ee
Therefore, either $|er(f^+,\xi,\La)| \ge \kappa_m(W)/2$ or $|er(f^-,\xi,\La)| \ge \kappa_m(W)/2$. Using Property A, we complete the proof.

 \end{proof}
 
 It is known (see \cite{VTbookMA}, p.264) that for $\mu$ being the Lebesgue measure on $[0,2\pi]^d$ we have
 \be\label{C6}
 \kappa_m(\bW^r_2) \ge C(r,d) m^{-r}(\log m)^{(d-1)/2}.
 \ee
 Inequality (\ref{C6}) and Theorem \ref{CT2} imply the lower bound in Theorem \ref{IT2} from Introduction. 
 
 We now make a comment on discretization on classes $\bE^r$ defined above at the end of Section \ref{sc}. The following lower bound is known (see \cite{Sha})
 \be\label{C7}
 \kappa_m(\bE^r) \ge C(r,d)m^{-r}(\log m)^{d-1}.
 \ee
 Inequality (\ref{C7}) and Theorem \ref{CT2} imply the lower bound
 \be\label{C8}
 er_m^o(\bE^r,L_2) \ge C(r,d)m^{-r}(\log m)^{d-1}.
 \ee
 Combining the lower bound (\ref{C8}) with Theorem \ref{AT2} we obtain the following result.
 \begin{Theorem}\label{CT3}  Let $r>1$ and $\mu$ be the Lebesgue measure on $[0,2\pi]^d$. Then
$$
er_m^o(\bE^r,L_2) \asymp  m^{-r}(\log m)^{d-1}.
$$
\end{Theorem}

We now make a remark on discretization of the $L_q$ norm for $q$ of the form $q=2^s$, $s\in \N$.  Introduce one more property.

{\bf Property As.} Let $s\in \N$. For any $k=0,\dots,s-1$ there exists a positive constant $c_k$ such that for all $f\in W$ we have $f_k^+ := c_k(f^{2^k}+1) \in W$ and $f_k^- := c_k(f^{2^k}-1) \in W$.

The above Property A corresponds to Property As with $s=1$ and $c_0=1/2$. In particular, a real symmetric class $W\subset \C(\Omega)$, which is convex, has quasi-algebra property and $1\in W$, satisfies Property As. 

\begin{Theorem}\label{CT4}   Let $q=2^s$, $s\in\N$. Suppose $W\subset \C(\Omega)$ has Property As. Then for any $m\in \N$ we have
$$
er_m^o(W,L_q) \ge c(s)\kappa_m(W),\qquad c(s)=c_0\cdots c_{s-1}.
$$
\end{Theorem}
\begin{proof} Note that for a real class $W$ and for even $q$ we have 
\be\label{C9}
er_m^o(W,L_q) = \kappa_m((W)^q),\qquad (W)^q :=\{f^q:\,f\in W\}.
\ee
Using our assumptions on $W$ we obtain that there is a $c_{s-1}>0$ such that for $f\in W$ we have $c_{s-1}(f^{q/2} +1)  \in W$ and $c_{s-1}(f^{q/2} -1) \in W$.
It is clear that repeating the argument of Theorem \ref{CT2} we obtain the following  inequality
\be\label{C10}
er_m^o(W,L_q)\ge c_{s-1}er_m^o(W,L_{q/2}).
\ee
This inequality combined with Theorem \ref{CT2} completes the proof.

\end{proof}

We note that an analog of Theorem \ref{CT4} holds for all even integers $q$. 

\begin{Theorem}\label{CT4'}   Let $q=2n$, $n\in\N$. Suppose a real symmetric class $W\subset \C(\Omega)$ is convex, has quasi-algebra property and $1\in W$. Then for any $m\in \N$ we have
$$
er_m^o(W,L_q) \ge c(a,n)\kappa_m(W),\qquad c(a,n)>0.
$$
\end{Theorem}
\begin{proof} Using our assumptions on $W$ we obtain that for any polynomial $P$ there exists $c(a,P)>0$ such that for any $f\in W$ we have $c(a,P)P(f)\in W$. Also observe that the operation $P(f+1)-P(f)$ eliminates the term $f^k$ with the highest degree $k$. Using these facts, arguing in the same way as in the proofs of Theorems \ref{CT2} and \ref{CT4}, we estimate $\kappa_m(W)$ from above by $C(a,n)er_m^o(W,L_q)$. 

\end{proof}

Combining Theorem \ref{CT4'} with Remark \ref{AR1} and using known results on numerical integration of classes $\bW^r_2$ and $\bE^r$ cited above we obtain the following two theorems.

\begin{Theorem}\label{CT5} Let $r>1/2$ and $\mu$ be the Lebesgue measure on $[0,2\pi]^d$. Then for $q=2n$, $n\in\N$, we have
$$
er_m^o(\bW^r_2,L_q) \asymp m^{-r}(\log m)^{(d-1)/2}.
$$
\end{Theorem}

 \begin{Theorem}\label{CT6}  Let $r>1$ and $\mu$ be the Lebesgue measure on $[0,2\pi]^d$. Then for $q=2n$, $n\in\N$, we have
$$
er_m^o(\bE^r,L_q) \asymp  m^{-r}(\log m)^{d-1}.
$$
\end{Theorem}

 {\bf Acknowledgement.} The work was supported by the Russian Federation Government Grant N{\textsuperscript{\underline{o}}}14.W03.31.0031. The paper contains results obtained in frames of the program 
 "Center for the storage and analysis of big data", supported by the Ministry of Science and High Education of Russian Federation (contract 11.12.2018N{\textsuperscript{\underline{o}}}13/1251/2018 between the Lomonosov Moscow State University and the Fond of support of the National technological initiative projects).

\end{document}